\documentclass[12pt]{article}
\usepackage[hyperfootnotes=false]{hyperref}  
\usepackage{amscd}     
\usepackage{amsmath}   
\usepackage{amssymb}   
\usepackage{amsthm}    
\usepackage{thmtools}
\usepackage{thm-restate}
\usepackage{bookmark}  
\usepackage{mathtools} 
\usepackage{epstopdf}  
\usepackage{verbatim}  
\usepackage{graphicx}  
\usepackage{color}     
\usepackage[skip=2pt,font={small, it}]{caption} 
\usepackage{algorithm} 
\usepackage{algorithmicx}
\usepackage{algpseudocode}
\usepackage{multirow}
\usepackage[normalem]{ulem}
\usepackage[useregional]{datetime2}
\usepackage[UKenglish]{babel}
\usepackage[english]{isodate}
\usepackage{setspace}

\usepackage[a4paper, total={7in, 10in}]{geometry}
\usepackage{geometry}

\graphicspath{{./Figures/}}

\newtheoremstyle{break}
  {\topsep}{\topsep}%
  {\itshape}{}%
  {\bfseries}{}%
  {\newline}{}%
\theoremstyle{break}

\newtheorem{Theorem}{Theorem}
\newtheorem{Lemma}[Theorem]{Lemma}

\newtheorem{corollary}[Theorem]{Corollary}

\theoremstyle{definition}

\newcommand{\RR}{\mathbb{R}}

\newcommand{\OO}{\mathcal{O}}

\newcommand{\defeq}{\stackrel{\textrm{def}}{=}}
 
\newcommand{\abs}[1]{\left\vert #1 \right\vert}
\newcommand{\norm}[1]{\left\Vert #1 \right\Vert}

\DeclareMathOperator*{\argmin}{\arg\!\min}
\DeclareMathOperator*{\argmax}{\arg\!\max}

\begin{document}
\title{Approximating the Span of Principal Components via Iterative Least-Squares}
\author{Yariv Aizenbud\footnote{authors contributed equally} ~and Barak Sober$^*$ 
}

\maketitle

\begin{abstract}
In the course of the last century, Principal Component Analysis (PCA) have become one of the pillars of modern scientific methods. 
Although PCA is normally addressed as a statistical tool aiming at finding orthogonal directions on which the variance is maximized, its first introduction by Pearson at 1901 was done through defining a non-linear least-squares minimization problem of fitting a plane to scattered data points. Thus, it seems natural that PCA and linear least-squares regression are somewhat related, as they both aim at fitting planes to data points. 
In this paper, we present a connection between the two approaches. Specifically,  we present an iterated linear least-squares approach, yielding a sequence of subspaces, which converges to the space spanned by the leading principal components (i.e., principal space).
\end{abstract}

\noindent\textbf{keywords:} Least-Squares, Eigensystem, Principal Component Analysis, Singular Value Decomposition, Iterative Least-Squares, Subspace Iterations

\section{Introduction}
Principal Component Analysis (PCA) is perhaps one of the most widely used algorithms in the last half a century. In almost every statistical analysis of high dimensional data, a preliminary step is to perform PCA and inspect its leading Principal Components (PCs). There is a vast literature regarding the theory and practices of PCA as well as the various ways to compute it (to name just a few works on this topic  \cite{chatelin2012eigenvalues,  daultrey1976principal,dunteman1989principal, gourlay1973computational,jolliffe2011principal, stewart2001matrix} and its applications \cite{aizenbud2015OutOfSample,friedman2001elements,novembre2008genes,ringner2008principal}). 
In this paper we present a new way to find the space spanned by the leading PCs of a given matrix through an iterative least-squares procedure. 
Even-though, computationally, it is not more efficient than the alternatives, this approach leads to a new and useful geometrical interpretation of the omnipresent PCA. 

Much of the fame PCA gained over the years results from its wide utilization in general scientific investigations. 
Explicitly, the fact that the leading PCs are the directions at which the variance of the data set is maximized, helps scientists explain measured phenomena. 
However, the discovery of PCA is usually attributed to Pearson in a paper published at 1901 \cite{Pearson1901PCA} that aimed at finding the best fitting plane to a given data set in the least-squares error sense (see Figure \ref{fig:PCvsLS}a). The fact that these seemingly different properties are  facets of the same actuality is discussed in more details in Section \ref{sec:SVDPCA}.
The resemblance between our iterated least-squares procedure and PCA arises from Pearson's presentation of the problem.  

The main contribution of this article is the derivation of an iterative least-squares algorithm to approximate the leading \emph{principle space} (the space spanned by the leading PCs of a data set), as well as proving that it has linear convergence rate. In mathematical terms, let $R = \lbrace r_i \rbrace_{i=1}^N$ be a set of data points sampled from $\RR^p$, we wish to approximate the linear space spanned by the $d$ leading PCs of $R$, through an iterative procedure. 
Given some initial guess of a $d$-dimensional basis $U_0$, we construct a sequence of bases $U_k$ whose span approximate the $d$-dimensional leading principle space.
We define $U_k$ by the following steps:
\begin{enumerate}
    \item Use $U_{k-1}$ as the axes of the ``$x$-domain"  (the independent variables) and its orthogonal complement as the ``$y$-domain"(the dependent variables).
    \item Perform linear regression (without a constant term).
    \item Pick some orthonormal basis $U_k$ for the linear space coinciding with the resulting linear approximation. \item Start over from Step 1 until convergence.
\end{enumerate}
See Section \ref{sec:IterativeLS} for a more formal presentation of the iterative procedure and the resulting Algorithm \ref{alg:LS2PC}.

Given $U_{k-1}\simeq \RR^d$ some $d$-dimensional subspace of $\RR^p$, the data set $R$ can be viewed as a function $f:U_{k-1}\rightarrow U_{k-1}^\perp$, where $U_{k-1}^\perp\simeq\RR^{p-d}$ is the orthogonal complement of $U_{k-1}$. 
Using statistical terminology, we treat the $p-d$ target coordinates as dependent on the $d$ domain coordinates.
Then, performing linear regression, without the constant term, will yield the directions minimizing the sum of squared distances from the regression plane only with respect to the target domain $U_{k-1}^\perp$ (see Figure \ref{fig:PCvsLS}a). 
On the other hand, computing the leading PCs of the data $R$ does not assume any functional relation between the coordinates; i.e., there are no independent and dependent variables, there is only a $p$ dimensional distribution.
Accordingly, the span of the leading PCs yields the plane minimizing the sum of squared Euclidean distances (see Figure \ref{fig:PCvsLS}b). 
As can be seen in Figure \ref{fig:PCvsLS}c, these two minimization problems yield different results.

The iterated least-squares approach we propose (which is described formally in Algorithm \ref{alg:LS2PC} below) can be summarized as follows.
First, given some initial coordinate system, perform a least-squares linear regression (see Figure \ref{fig:LS_Iterates}a), without the intercept (constant term). 
Second, take the resulting linear approximation as a new $x$-domain and its orthogonal complement as the new $y$-domain (Figure \ref{fig:LS_Iterates}b).
In other words, rotate the axes to the point where the linear approximation coincides with the $x$-axis, and repeat the regression-rotation sequence until convergence (Figure \ref{fig:LS_Iterates}c-d).

As stated above, taking these iterations to the limit yields the space spanned by the $d$ leading PCs. 
That is, the non-linear least-squares problem of finding the plane minimizing the Euclidean distances from the scattered data, can be solved through a sequence of linear least-squares minimizations.

\begin{figure}[ht]
\begin{centering}
\includegraphics[width={1.0\linewidth}]{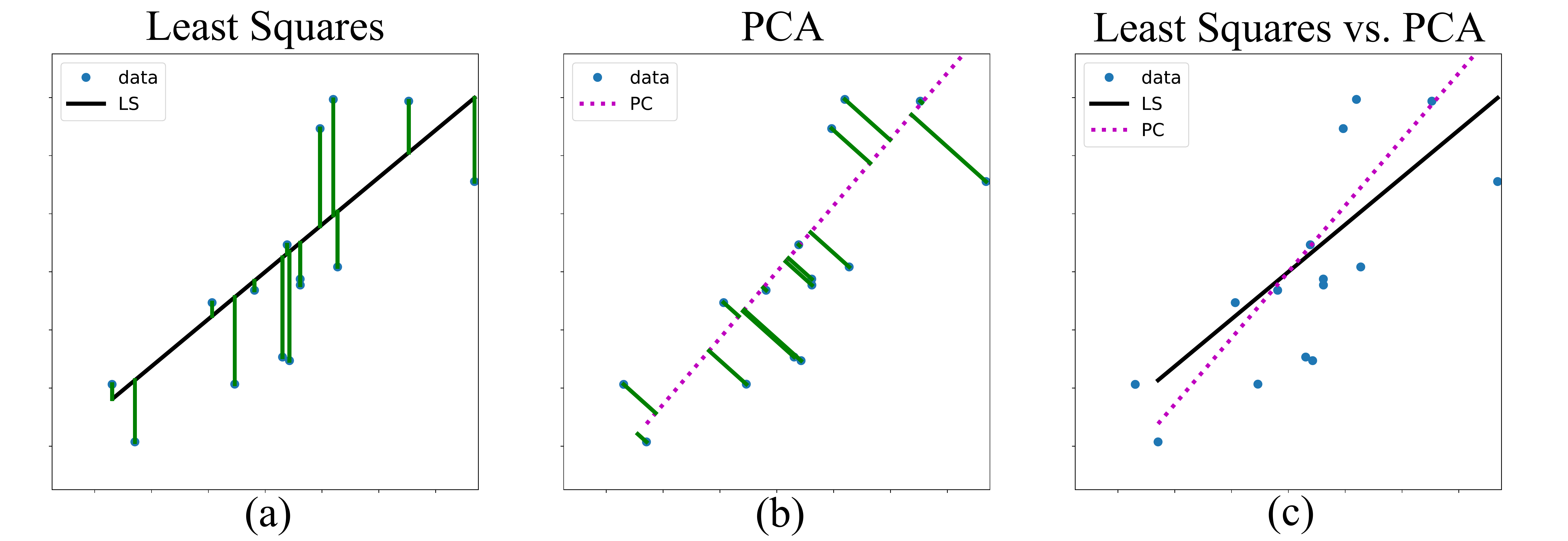}
\par\end{centering}

\caption{Least-Squares approximation compared with Principal Component Analysis (PCA) - a 2D example: \emph{(a)} a least-squares (LS) approximation of a given sample set (the errors which are minimized are marked in green); (b) the first Principle Component (PC) for the same data points (the errors which are minimized are marked in green); (c) LS and PC approximation overlaid.
}
\label{fig:PCvsLS}
\end{figure}

\begin{figure}[ht]
\begin{centering}
\includegraphics[width={1.0\linewidth}]{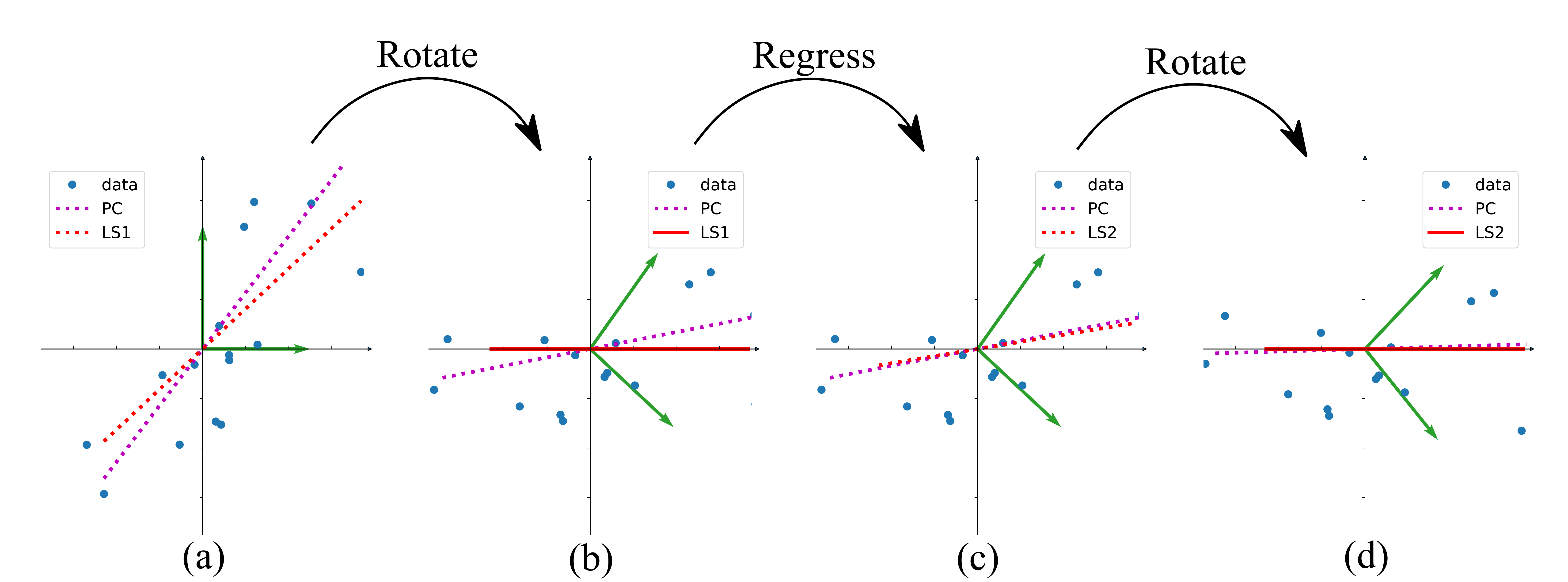}
\par\end{centering}
\caption{Two iterations of Algorithm \ref{alg:LS2PC} applied to data sampled from $\RR^2$. In the leftmost image we portray the sampled data (the data points are marked in blue), and their regression line (in dashed red) with respect to some initial coordinate system (x and y axes are marked in green). The following image (second from left) is a rotation of the same data to a point that the regression line coincide with the horizontal axis. The green arrows mark the original coordinate system to emphasize the rotation of the data. The third image from left is the same as the former just with an added regression line with respect to the current coordinate system (not the original xy axes). The rightmost sub-figure is a rotated version of its preceding sub-figure; as before, the regression line is aligned with the horizontal axis.
\label{fig:LS_Iterates}
}
\end{figure}

In what follows, we give a brief introduction to PCA and explain its relationship to a least-squares problem. 
Then, we describe shortly how one can compute the Principal Spaces through Subspace Iterations.
Subsequently, we present our own method of computing the Principal Spaces, which essentially can be understood as a new Least-Squares equivalent formalization of the Subspace Iterations. 
For a more comprehensive treatment of the topic of principle space or eigenspace computation, we refer the reader to some classical literature on this subject \cite{chatelin2012eigenvalues, gourlay1973computational, stewart2001matrix, trefethen1997numerical}.
\section{Preliminaries}
\subsection{PCA and its relation to the eigenproblem}
\label{sec:SVDPCA}

In the following passages, we use a very elementary language to bridge between the statistical concept of variance maximization in PCA, and the non-linear least-squares problem described by Pearson (i.e., the minimization of the sum of Euclidean squared distances). 
This is, a very well known fact (e.g., see \cite{mirsky1960symmetric}), which reformulates the PCA as an eigenproblem of symmetric positive semi-definite matrices, and we do not claim to be original in connecting these ideas. 
Nonetheless, the approach we take below, which we have not encountered in the literature, is self-contained and does not require any knowledge beyond basic algebraic concepts such as the orthogonal projection.

Starting with the statistical perspective of PCA, our sample set, $\lbrace r_i \rbrace_{i=1}^N$, consists of $N$ samples of a $p$-dimensional random vector $\vec{x}$ with zero mean. Then, the leading PC is defined as the direction explaining most of the variance in the data.
Explicitly, by projecting the data onto this direction we would get the one dimensional random variable with maximal sample variance; i.e.,
\begin{equation}\label{eq:PC1} 
\textrm{PC}_1 \defeq \argmax_{\substack{
		\norm{w}=1 \\
		w\in\RR^p}}
	\sum_{i=1}^{N} \abs{w^T \cdot r_i}^2 = 
	\argmax_{\substack{
			\norm{w}=1 \\
			w\in\RR^p}}
		\norm{w^T \cdot \mathcal{R}}^2
,\end{equation}
where $ \mathcal{R} $ is a matrix  whose columns are $ r_i $ and $ \norm{\cdot} $ is the standard Euclidean norm. 
Then, for $ k > 1 $ we define by recursion 
\begin{equation}\label{eq:PCk}
\textrm{PC}_{k} \defeq \argmax_{\substack{
						\norm{w}=1 \\
						w\perp \mathcal{PS}_{k-1}}}
 \norm{w^T \cdot {\mathcal{R}}}^2
,\end{equation}
where
\begin{equation}
\mathcal{PS}_{k-1} \defeq span\{\textrm{PC}_1, \ldots \textrm{PC}_{k-1}\}.
\end{equation}
Naturally, these definitions result with the formation of an orthonormal basis $ \{\textrm{PC}_1, ... , \textrm{PC}_p\} $.
We denote by $U$ the matrix whose columns are $\{\textrm{PC}_j\}_{j=1}^p$. Considering our data matrix $\mathcal{R}$ in this new basis,
\begin{equation}\label{eq:defYj}
  Y \defeq U^T \cdot \mathcal{R}  
,\end{equation}
reorganizes the variables in a way such that the sample covariance matrix is diagonal.

The sample covariance matrix of $\mathcal{R}$ can be written as 
\[
C = \mathcal{R} \mathcal{R}^T \cdot \frac{1}{N}
.\]
Furthermore, $ \mathcal{R} \mathcal{R}^T $ is symmetric and positive semi-definite and its orthogonal diagonalization is
\begin{equation}\label{eq:RRtDiag}
C = U D U^T
,\end{equation}
where $D = diag(\sigma_1^2, \ldots, \sigma_p^2)$.
For simplicity, we assume\footnote{if we neglect this assumption, the behavior of principle  spaces, which is in the focus of this chapter, is maintained, but the uniqueness of the principal components and their order is lost\label{fnt:1}} that the eigenvalues $ \sigma_1^2 > \sigma_2^2 > \ldots > \sigma_p^2 > 0$ are distinct, and get that $u_j$ and $PC_j$ spans the same subspace.
From \eqref{eq:defYj} we have
\begin{equation}\label{eq:PCvariance}
\sigma_j^2 = \text{var}(Y^j) = \frac{1}{N}\sum_{i=1}^N\abs{\langle \textrm{PC}_j , r_i\rangle}^2,
\end{equation}
where $Y^j$ is the j-th row of $Y$, and $\text{var}(\cdot)$ denotes the sample variance. 

Let us now turn to the Singular Value Decomposition (SVD) \cite{stewart2001matrix}  of the matrix $ \mathcal{R} $. The matrix $\mathcal{R}$ can be written as
\[
\mathcal{R} = U \Sigma V^T
,\]
where $ U, V $ are orthogonal matrices and $ \Sigma = diag(\tilde{\sigma}_1, ... , \tilde{\sigma}_n) $ is a diagonal matrix with the singular values on the diagonal.
Accordingly, we can pronounce
\begin{equation}\label{eq:UPCs}
C = \mathcal{R} \mathcal{R}^T \cdot \frac{1}{N} = U \Sigma \Sigma^T U^T \cdot \frac{1}{N}
,
\end{equation}
and the singular values $\tilde{\sigma}_j$ maintain
\begin{equation}
\tilde{\sigma}_j^2 = N \cdot\text{var}(Y^j) = N\cdot \sigma_j^2
,\end{equation}
which by equation \eqref{eq:PCvariance} equals
\begin{equation}\label{eq:SigmaVar}
\tilde{\sigma}_j^2 = \sum_{i=1}^N\abs{\langle \textrm{PC}_j , r_i\rangle}^2
\end{equation} 

To summarize, up until this point, we have shown the connection between SVD, eigen-decompostion and the variance maximization aspect of PCA.
Turning to the non-linear least-squares problem, introduced by Pearson, we denote by $dist(r_i, w)^2$ the squared Euclidean distance between each point $r_i$ and the line spanned by the vector $w$, and achieve (assuming $ \norm{w}=1 $)
\[\arraycolsep=1.4pt\def\arraystretch{1.8}
\begin{array}{ll}
     \sum_{i=1}^N dist(r_i, w)^2 &= \sum\limits_{i=1}^N\norm{r_i - \langle w , r_i\rangle w}^2  \\
     &= \sum\limits_{i=1}^N \norm{r_i}^2 -  \abs{\langle w, r_i\rangle}^2 \\
     &= \sum\limits_{i=1}^N\norm{r_i}^2 -  \sum\limits_{i=1}^N\abs{\langle w, r_i\rangle}^2
\end{array}
,\]
or, in matrix notation,
\begin{equation}\label{eq:Sumdiw}
\sum_{i=1}^N dist(r_i, w)^2 = \norm{\mathcal{R}}_{\text{F}}^2 -  \sum_{i=1}^N\abs{\langle w, r_i\rangle}^2
\end{equation}
where $ \norm{\mathcal{R}}_{\text{F}} $ denotes the Frobenius norm of the matrix $ \mathcal{R} $.
Using \eqref{eq:SigmaVar} we get
\begin{equation}
\sum_{i=1}^N dist(r_i, \textrm{PC}_j)^2 = 
\norm{\mathcal{R}}_{\text{F}}^2 - \tilde{\sigma}_j^2.
\end{equation}
Thus, from \eqref{eq:PC1} and \eqref{eq:PCk} we have
\begin{equation}
\arraycolsep=1.4pt\def\arraystretch{2.0}
\begin{array}{ll}
\sum\limits_{i=1}^N dist(r_i, \textrm{PC}_1)^2 &= \norm{\mathcal{R}}_{\text{F}}^2 -
\max\limits_{\substack{
		\norm{w}=1 \\
		w\in\RR^p}}
\sum\limits_{i=1}^{N} \abs{\langle w , r_i \rangle}^2 \\
     &=\min\limits_{\substack{
		\norm{w}=1 \\
		w\in\RR^p}} 
\sum\limits_{i=1}^{N}\norm{ r_i - \langle w , r_i \rangle w}^2\\
 &= \min\limits_{\substack{
		\norm{w}=1 \\
		w\in\RR^p}} 
\sum\limits_{i=1}^N dist(r_i, w)^2,
\end{array}
\end{equation}
and similarly for $ k > 1 $
\begin{equation}
\sum_{i=1}^N dist(r_i, \textrm{PC}_k)^2 = 
\min_{\substack{
		\norm{w}=1 \\
		w\perp \mathcal{PS}_{k-1}}} 
\sum_{i=1}^N dist(r_i, w)^2
.
\end{equation}
As a result, we can see that the PCs can be defined equivalently as
\begin{equation}\label{eq:PC1min}
\textrm{PC}_1 \defeq
\argmin_{\substack{
		\norm{w}=1 \\
		w\in\RR^p}} 
\sum_{i=1}^N dist(r_i, w)^2 
\end{equation}
\begin{equation}\label{eq:PCkmin}
\textrm{PC}_k \defeq
\argmin_{\substack{
		\norm{w}=1 \\
		w\perp \mathcal{PS}_{k-1}}} 
\sum_{i=1}^N dist(r_i, w)^2 
\end{equation}

In other words, the variance maximization property of the PCs (expressed in the definitions of Equations \eqref{eq:PC1}-\eqref{eq:PCk}) is equivalent to the fact that they minimize the sum of squared distances to the sample set. 
As mentioned above, this result can be derived directly from the Eckart-Young-Mirsky theorem \cite{eckart1936approximation, mirsky1960symmetric}. 
Nevertheless, the aforementioned derivation gives a clearer motivation to the utilization of least-squares in the process of calculating the Principal Space (i.e., the space spanned by the PCs). 

In addition, we can see from Equation \eqref{eq:RRtDiag} that the  PCs are just eigenvectors of $ \mathcal{R} \mathcal{R}^T $. Thus, finding the span of PCs is equivalent to finding eigenspaces of the symmetric matrix $ \mathcal{R} \mathcal{R}^T $. 

\subsection{Subspace Iterations}
As shown above, finding the Principal Space of a matrix is equivalent to finding the eigenspace of a symmetric positive semi-definite matrix. 
In general, the problem of finding the eigenvalues of a given matrix is equivalent to finding roots of an $n^{th}$ degree polynomial \cite{stewart2001matrix}.
Thus, extracting the eigenvalues and eigenvectors of a given matrix of dimension $n$ cannot have a closed solution, and is bound to be iterative.

Most of the known approaches to solve the eigenproblem are some generalizations of the Power Iterations \cite{stewart2001matrix,trefethen1997numerical}.
The basic idea behind this algorithm is that if we take some initial vector $u$,  repeatedly apply $A$ on it, and normalize the result, then we get in the limit the most dominant eigenvector (under the simplifying assumption that $\lambda_1 > \lambda_2$; see Algorithm \ref{alg:Power} for more details). 
For a more rigor and comprehensive treatment see \cite{stewart2001matrix} for example.

\begin{algorithm}
\caption{Power Iterations}
\label{alg:Power}
\begin{algorithmic}[1]
\State {\bfseries Input:} $A\in\RR^{p\times p}, u\in\RR^p$, where $u$ is a first guess for the leading eigenvector (can be picked at random).
\State{\bfseries Output:} $v\in\RR^p$ - the leading eigenvector of $A$.
\State $v = u$
\Repeat
    \State $v_{prev} = v$
    \State $z = A v$
	\State $v = \frac{z}{\norm{z}}$
\Until {$\|v-v_{\text{prev}}\|<\epsilon$}
\end{algorithmic}
\end{algorithm}

In this paper, we are concerned with computing the span of the leading $d$ eigenvectors (assuming that there is a gap between the $d$ and $d+1$ eigenvalues). 
One of the known approaches to tackle this problem is the Subspace Iterations, which is a generalization of the Power Iterations.
Similar to Power Iterations, these iterations comprise the multiplication of the matrix $A$ with the former approximating basis of the subspace and performing Gram-Schmidt procedure (G-S) to yield the next approximating basis (for more details see Algorithm \ref{alg:Subspace}).
Although this is a naive version of the Subspace Iteration method, under the assumptions that $A$ is symmetric and that 
$\lambda_1 \geq \lambda_2 \geq \ldots \geq \lambda_d > \lambda_{d+1} \geq \ldots \geq \lambda_n$  the algorithm is promised to converge exponentially fast (i.e., geometric rate of convergence) \cite{stewart2001matrix}. 
Below, we show that the Iterative Least-Squares algorithm presented here (Algorithm \ref{alg:LS2PC}) coincides with this version of the Subspace Iterations. Thus, Algorithm \ref{alg:LS2PC} also converges to the Principle Space, with the same convergence rates. 

\begin{algorithm}
\caption{Subspace Iterations}
\label{alg:Subspace}
\begin{algorithmic}[1]
\State {\bfseries Input:} $A\in\RR^{p\times p}, U\in\RR^{p\times d}$, where $U$ is a first guess basis for the leading $d$ dimensional eigenspace.
\State{\bfseries Output:} $V\in\RR^{p\times d}$ - a basis for the $d$-dimensional eigenspace of $A$.
\State V = U
\Repeat
    \State $V_{prev} = V$
	\State $Z = A V$
	\State $V = \textrm{G-S}(Z)$
\Until {$\|V-V_{\text{prev}}\|<\epsilon$}
\end{algorithmic}
\end{algorithm}

\section{Iterative least-squares}
\label{sec:IterativeLS}
Let $\mathcal{R}$ be a matrix of dimension $p\times N$ and we wish to compute its $d$-dimensional Principal Space; i.e, the span of its $d$ leading PCs.
Let the principal values of $\mathcal{R}$ be $\sigma_1\geq\sigma_2\geq\cdots\geq\sigma_d>\sigma_{d+1}\geq\cdots\geq\sigma_p$, and let $\{r_i\}_{i=1}^N$ denote the columns of $\mathcal{R}$,
\[
\mathcal{R} = 
\left(
\begin{array}{ccc}
| &  & | \\
r_1 & \cdots & r_N \\
| &  & | 
\end{array}
\right)_{p\times N}
.\]
Then, given an initial $d$-dimensional orthonormal coordinate system $ \{u_j^0\}_{j=1}^d $, written in matrix form as
\[
U_0 = 
\left(
\begin{array}{ccc}
| &  & | \\
u^0_1 & \cdots & u^0_d \\
| &  & | 
\end{array}
\right)_{p\times d}
,\]
we iterate the following two-steps procedure.

\begin{enumerate}
	\item Define $X_k$ to be the projections of $r_i$ onto $Col(U_k)$, the column space of $U_k$:
	    \[
        X_k = 
        \left(
        \begin{array}{ccc}
        | &  & | \\
        x^k_1 & \cdots & x^k_N \\
        | &  & | 
        \end{array}
        \right)_{d\times N}=
        U_k^T \cdot \mathcal{R},
        \] 
        and solve the linear least-squares problem
        \begin{equation}\label{eq:LSiter}
        A_{k+1} = \argmin_{A\in \RR^{p\times d}} \sum_{i=1}^N \norm{r_i - A x_i^k}^2 = \argmin_{A\in \RR^{p\times d}}\norm{\mathcal{R} - A X_k}_F^2,
        \end{equation}
    	
	\item Apply Gram-Schmidt on the columns of $A_{k+1}$ to get a new orthogonal coordinate system. Namely,
	\begin{equation}\label{eq:LSiter_Step2}
	U_{k+1} \defeq \textrm{G-S}(A_{k+1})
	,\end{equation}
	where $\textrm{G-S}(A)$ stands for the Gram Schmidt process applied on the columns of a matrix $A$, yielding an orthonormal basis to the column space of $A$.
\end{enumerate}

The only demand in Power iterations, needed to show its convergence (given that the first two eigenvalues are distinct), is that the initial vector is not perpendicular to the direction of the most dominant eigenvector.
Similarly, to prove the convergence of Algorithm \ref{alg:LS2PC} we generalize this requirement to
\begin{equation}\label{eq:TheCondition}
  rank(U_0^T \cdot\mathcal{U}_d) = d,  
\end{equation}
where $U_0$ denotes the initial guess for a basis and $\mathcal{U}_d$ is the matrix
\[
\mathcal{U}_{d} =
\left(
\begin{array}{ccc}
     | & & |  \\
     PC_1 & \cdots & PC_d \\
     | & & |
\end{array}
\right)
.\]
\begin{algorithm}
\caption{Iterative Least-Squares - LS2PC}
\label{alg:LS2PC}
\begin{algorithmic}[1]
\State {\bfseries Input:}\begin{tabular}[t]{ll} 
$\lbrace r_i \rbrace_{i=1}^N\subset\RR^p$ & Data set.   \\ 
$\epsilon\in\RR$ & Precision threshold.\\
$d$ & Subspace dimension.\\
$U_0$ & $p\times d$ orthogonal matrix - initial guess.
\end{tabular}
\State{\bfseries Output:} $U$ - a $p\times d$ matrix with columns that span the $d$ largest PCs
\State define $ \mathcal{R} $ to be a $ p\times N $ matrix whose columns are $ \{r_i\}_{i=1}^N $
\Repeat
	\State $ U_{\text{prev}} = U $
	\State $X = U^T \mathcal{R}$ \Comment{Change the coordinate system }
	\State Solve $A \cdot X X^T = \mathcal{R} \cdot X^T$ for $A$ \Comment{Solve least squares problem}
	\State $U = \textrm{G-S}(A)$ \Comment{Find orthogonal basis}
\Until {$\|U-U_{\text{prev}}\|<\epsilon$}
\end{algorithmic}
\end{algorithm}

We now wish to prove the following theorem:
\begin{Theorem}[Iterative least-squares is equivalent to subspace iterations]\label{thm:ILS_eq_SI}
Let $\mathcal{R}$ be a matrix of dimensions $p\times N$ with $N>p$ and singular values $\sigma_1 \geq \sigma_2 \geq \cdots \sigma_d > \sigma_{d+1} \geq \cdots \geq\sigma_p$. 
Let $U_0$ be a matrix of dimension $p\times d$, satisfying condition \eqref{eq:TheCondition} (i.e., $rank(U_0^T\cdot \mathcal{U}_d) = d$).
Initializing with $U_0$ Algorithm \ref{alg:LS2PC} for $\mathcal{R}$ and Algorithm \ref{alg:Subspace} for $\mathcal{R}\mathcal{R}^T$ will result in $U^{LS}_k$ and $U^{SI}_k$ respectively ($k\geq 0$), where
\begin{equation}
    Col(U_k^{LS}) = Col(U_k^{SI})
\end{equation}
\end{Theorem}

\begin{proof}
The outline of the proof is as follows, under the conditions of the theorem we can show 
\begin{enumerate}
    \item \label{outline:LSisSI}$Col(U_1^{LS}) = Col({U}^{SI}_1)$.
    \item \label{outline:U1Invert}${U}^{LS}_1$ satisfies condition \eqref{eq:TheCondition}.
\end{enumerate}
Then, by the same means (by intializing $U_0$ with $U_1^{LS}$) we can show that $Col(U_2^{LS}) = Col({U}^{SI}_2)$ and ${U}^{LS}_2$ satisfies condition \eqref{eq:TheCondition}.
Thus, the theorem is given by induction.

Lemma~\ref{lem:XXtInvertible} below, shows that, under the conditions of the theorem, $X_0X_0^T = U_0^T \mathcal{R}\mathcal{R}^TU_0$ is invertible, thus, by Lemma~\ref{lem:AkSolution} applied on $A_1$ from \eqref{eq:LSiter}, we have that 
\begin{equation*}
A_{1} = \mathcal{R} \mathcal{R}^T U_0 (U_0^T \mathcal{R} \mathcal{R}^T U_0)^{-1}.
\end{equation*}
Furthermore, as $(U_0^T \mathcal{R} \mathcal{R}^T U_0)^{-1}\in\RR^{d\times d}$ is of full rank, we get that 
\begin{equation}\label{eq:col_space_eq}
    Col(U_{1}^{LS}) = Col(A_{1}) = Col(\mathcal{R} \mathcal{R}^T U_0) = Col({U}^{SI}_1).
\end{equation}

We now turn to show claim \ref{outline:U1Invert} of the outline; namely the fact that 
\[
rank(\mathcal{U}_d^T U_1^{LS}) = d
,\]
which by \eqref{eq:col_space_eq} can be pronounced equivalently as
\[
rank(\mathcal{U}_d^T U_1^{SI}) = d
.\]
Note that by definition \eqref{eq:LSiter_Step2}
\[
Col(U_1^{SI}) = Col(\mathcal{R}\mathcal{R}^T U_0)
,\]
and so, to prove the claim suffice it to show that
\[
rank(\mathcal{U}_d^T \mathcal{R}\mathcal{R}^T U_0) = rank((\mathcal{R}\mathcal{R}^T U_0)^T \mathcal{U}_d) = d
.\]
Denote by $\overline{\mathcal{U}}_d$ the $(p-d) \times p$ matrix with columns that complete the columns of $\mathcal{U}_d$ to an orthonormal basis of $\RR^p$. From \eqref{eq:UPCs} and from the fact that $\sigma_1\geq\sigma_2\geq\cdots\geq\sigma_d>\sigma_{d+1}\geq\cdots\geq\sigma_p\geq 0$, it is evident that
\begin{equation}\label{eq:RRT_spectral_decomp}
    \begin{array}{c}
\mathcal{R}\mathcal{R}^T \mathcal{U}_d = \mathcal{U}_d L_1\\ \mathcal{R}\mathcal{R}^T \overline{\mathcal{U}}_d = \overline{\mathcal{U}}_d L_2    
\end{array}
\end{equation}
for $L_1 \in \RR ^{d \times p}$ of full rank, and $L_2 \in \RR ^{(p-d) \times p}$. Additionally, note that since $[\mathcal{U}_d, \overline{\mathcal{U}}_d]$ is an orthonormal basis, we can write $U_0 = \mathcal{U}_d C_1 + \overline{\mathcal{U}}_d C_2 $. 
Thus, from \eqref{eq:RRT_spectral_decomp} we have 
\begin{equation}\label{eq:RRtU0_as_Sum}
    \mathcal{R}\mathcal{R}^TU_0 = \mathcal{U}_d L_1 C_1 + \overline{\mathcal{U}}_d L_2 C_2,
\end{equation}
and
\begin{equation}\label{eq:rankUU_0_eq_rankC_1}
d = rank(U_0^T \mathcal{U}_d) =rank(\mathcal{U}_d^T U_0) = rank(\mathcal{U}_d^T (\mathcal{U}_d C_1 + \overline{\mathcal{U}}_d C_2 ))=rank(C_1).
\end{equation}
Since $rank(L_1) =rank(C_1) = d$, it follows that $rank(L_1 C_1) = d$, and from \eqref{eq:RRtU0_as_Sum} we get
\[
rank((\mathcal{R}\mathcal{R}^T U_0)^T \mathcal{U}_d) = rank([C_1^T L_1^T\mathcal{U}^T_d  + C_2^T L_2^T \overline{\mathcal{U}}^T_d] \mathcal{U}_d) = rank(C_1^T L_1^T) = d
\]
In other words, we showed that $U^{SI}_1$ (or, equivalently $U^{LS}_1$) satisfies the conditions of this theorem (in the role of $U_0$), and the proof is concluded.
\end{proof}
Since Algorithm \ref{alg:LS2PC} is equivalent to subspace iterations, it inherits all the properties and improvements that have been developed for subspace iterations. One result that we mention explicitly in the next corollary which follows directly from the properties of subspace iterations (see for example \cite{stewart2001matrix}).
\begin{corollary}
Under the conditions of Theorem \ref{thm:ILS_eq_SI} 
\[
\lim_{k\rightarrow\infty}Col(U^{LS}_k) = Col(\mathcal{U}_d)
,\]
and the rate of convergence is $\OO\left(\left(\frac{\sigma^{2}_{d+1}}{\sigma^{2}_d} + \varepsilon\right)^k\right)$ for all $\varepsilon > 0$.
\end{corollary}

Lemma \ref{lem:AkSolution} is a simple extension of the standard ordinary least-squares (OLS) solution to multiple ``right hand sides" at once. The proof is presented for completeness.
\begin{Lemma}\label{lem:AkSolution}
Let 
$$
B = \argmin_{A\in M_{p\times d}}\norm{\mathcal{R} - A X}_F^2,
$$
and assume that $X X^T$ is invertivle. Than 
$$
B = \mathcal{R} X^T (X^TX)^{-1}.
$$
\end{Lemma}

\begin{proof}
    Note that,
	\begin{equation}\label{eq:B_opt_equivalent_form}
	    B = \argmin_{A\in M_{p\times d}} \sum_{i=1}^N \norm{r_i - A x_i}^2,
	\end{equation}
	where $r_i$ and $x_i$ are the columns of $\mathcal{R}$ and $X$ correspondingly.
	Since \eqref{eq:B_opt_equivalent_form} this is a minimization of a sum of squares, it can be solved row by row. Namely, by denoting $ a^j $ the $ \text{j}^{\text{th}} $ row of a matrix $ A $ we can write
	\[
	 \min_{A\in M_{p\times d}} \sum_{i=1}^N \norm{r_i - A x_i}^2  = 
	 \sum_{j=1}^{p}\min_{(a^j)^T\in \RR^d} \sum_{i=1}^N \abs{r_i^j - a^j\cdot x_i}^2
	,\]
	where $ r_i^j $ denotes the $ \text{j}^{\text{th}} $ coordinate of the vector $ r_i $. 
	In other words, each row $b^j$ of the desired minimizing matrix $ B $ of \ref{eq:B_opt_equivalent_form} can be decided independently by solving 
	\begin{equation}\label{eq:AkMinimization}
	(b^j)^T = \argmin_{w^T\in \RR^d} \sum_{i=1}^N \abs{r_i^j - w\cdot x_i}^2 = 
	\argmin_{w^T\in \RR^d} \norm{r^j - w X}^2  =
	\argmin_{w^T\in \RR^d} \norm{(r^j)^T - X^T w^T}^2 
	,\end{equation}
	where $ r^j $ is the $ \text{j}^{\text{th}} $ row of the matrix $ \mathcal{R} $.
	However, this is just a scalar linear least squares equation and assuming that $X X^T$ is invertible, it can be solved by
	\begin{equation}\label{eq:AkjSolution}
	(b^j)^T = (X \cdot X^T)^{-1}\cdot X \cdot (r^j)^T
	.\end{equation}
	Thus, we have
	\begin{equation}
	B^T = (X \cdot X^T)^{-1}\cdot X \cdot \mathcal{R}^T,
	\end{equation}
	or,
	\begin{equation}
	B = \mathcal{R} X^T (X^TX)^{-1}.
	\end{equation}
\end{proof}

As can be seen in the following lemma, the fact that $X_0 X_0^T$ is invertible is a result of the condition $rank(U_0^T \mathcal{U}_d)=d$.
In other words, as long as our basis $U_k$ is not orthogonal in any principal direction to the principal vectors $\mathcal{U}_d$  the matrix is indeed invertible. 

\begin{Lemma}\label{lem:XXtInvertible}
Under the notations of Algorithm \ref{alg:LS2PC}, if $rank(U_0^T \mathcal{U}_d)=d$ then the matrix $X_0 X_0^T\in \RR^{d\times d}$ is invertible.
\end{Lemma}
\begin{proof}
    In order to prove that $X_0 X_0^T\in\RR^{d\times d}$ is inevitable, it is enough to show that 
    \[
    rank(X_0) = rank(U_{0}^T \mathcal{R}) = d
    ,\]
    as $rank(X_0)=rank(X_0 X_0^T)$ for all matrices $X\in \RR^{d\times p}$.
    
    Since the columns of $\mathcal{U}_d$ lie in the image of $\mathcal{R}$, there is a matrix $V$ such that $rank (V) = d$ and $\mathcal{R}V = \mathcal{U}_d$. We note that
    $$
    d \geq rank(X_0) = rank(U_{0}^T \mathcal{R}) \geq rank(U_{0}^T \mathcal{R}V) = rank(U_{0}^T \mathcal{U}_d) = d,
    $$
    which concludes the proof.
\end{proof}

\section{Conclusion}
We have established a relation between the non-linear least-squares problem of Principal Component Analysis and a sequence of linear least-squares minimizations.
Apparently, the iterative least-squares algorithm presented above coincides with the well known Subspace Iterations algorithm.
As a consequence each iteration of Subspace (or Power) Iterations, can be interpreted as a solution to a least-squares problem.
In other words, solving a least-squares problem is equivalent to multiplying a basis with the sample covariance matrix.
Although our approach does not lead to a new way of computing the Principal Space, it can be used as a theoretical tool in analysis (e.g., see \cite{sober2017approximation,sober2016MMLS}).

\section*{Acknowledgments}
We would like to thank Prof. David Levin for the fruitful conversations he shared with us. B. Sober is supported through Math+X grant 400837 from the Simons Foundation. 
Y. Aizenbud is supported by the Israel Science Foundation (ISF, 1556/17),
Blavatnik Computer Science Research Fund, Indo-Israel Collaborative for Infrastructure Security - Israel Ministry of Science and Technology 3-14481, a Fellowship from Jyv\"{a}skyl\"{a} University and the Clore Foundation.

\bibliography{mybib}{}
\bibliographystyle{abbrv}

\end{document}